\documentclass[11pt,fleqn]{article}
\usepackage{latexsym}
\usepackage{amsmath, amsthm}
\usepackage{graphicx} 
\usepackage{amsfonts} 
\usepackage{amssymb}
\usepackage[margin=20mm,includefoot]{geometry}

\DeclareMathOperator{\oinv}{oinv}
\DeclareMathOperator{\onsp}{onsp}
\DeclareMathOperator{\nneg}{neg}
\DeclareMathOperator{\oneg}{oneg}
\DeclareMathOperator{\inv}{inv}
\DeclareMathOperator{\nsp}{nsp}

\DeclareMathOperator{\ensp}{ensp}
\DeclareMathOperator{\eneg}{eneg}
\DeclareMathOperator{\einv}{einv}
\DeclareMathOperator{\hgt}{ht}
\DeclareMathOperator{\Des}{Des}

\theoremstyle{plain}
\newtheorem{thm}{Theorem}[section]
\newtheorem{pro}[thm]{Proposition}
\newtheorem{lem}[thm]{Lemma}

\newtheorem{cor}[thm]{Corollary}

\theoremstyle{definition}
\newtheorem{defn}[thm]{Definition}

\theoremstyle{remark}
\newtheorem{rmk}[thm]{Remark}

\newcommand{\p}{\noindent}

\newcommand{\N}{\mathbb N}

\newcommand{\PP}{{\mathbb P}}

\begin{document}
	\title{Odd length in Weyl groups}\date{}\maketitle \begin{center}
		
		Francesco Brenti \\
		
		Dipartimento di Matematica \\
		
		Universit\`{a} di Roma ``Tor Vergata''\\
		
		Via della Ricerca Scientifica, 1 \\
		
		00133 Roma, Italy \\
		
		{\em brenti@mat.uniroma2.it } \\
		
		\vspace{0.5cm}
		
		Angela Carnevale \footnote{Partially supported by German-Israeli Foundation for Scientific Research and Development, grant no. 1246.} \\
		
		Fakult\"at f\"ur Mathematik \\
		
		Universit\"at Bielefeld\\
		
		D-33501 Bielefeld, Germany \\
		
		{\em acarneva1@math.uni-bielefeld.de } \\
	\end{center}
	
	\vspace{0.5cm}

\begin{abstract}
We define a new statistic on any Weyl group which we call the odd length and which reduces, for
Weyl groups of types $A$, $B$, and $D$, the the statistics by the same name that have already been 
defined and studied in \cite{KV}, \cite{VS}, \cite{VS2}, and \cite{BC2}. We show that the signed (by length) generating function of the odd length
always factors nicely except possibly in type $E_8$, and we obtain multivariate analogues of these
factorizations in types $B$ and $D$.	
\end{abstract}

\thispagestyle{empty}
\section{Introduction}

A new statistic on the symmetric group was introduced in \cite{KV} in relation to formed spaces.
This statistic combines combinatorial and parity conditions and is now known as the odd
length. Similar statistics were introduced and studied in \cite{VS} and \cite{VS2} in type $B$ in relation to local
factors of representation zeta functions of certain groups and in type $D$ in \cite{BC2} and \cite{CTh}.
The signed, by length, distribution of the odd length over quotients of the hyperoctahedral groups is also related to the enumeration of matrices satisfying certain properties and with 
fixed rank. In particular, the signed distribution of the odd length on the maximal quotients of $B_n$ is strictly related to the number of symmetric $n\times n$ matrices of given rank over finite fields (see \cite{VS} for a precise conjecture,  \cite{BC1} for a proof, and \cite{CSV} for related work). In \cite{KV} and \cite{VS2} closed product formulas were conjectured for the signed generating
function of the odd length over all quotients of the symmetric and hyperoctahedral groups, 
respectively. These conjectures were proved in  \cite{BC1} (see also \cite{CTh}) in types $A$ and $B$ and independently in \cite{Lan} for type $B$. 

In this paper we define a new statistic on any Weyl group. This statistic depends on the root system 
underlying the Weyl group and we compute it combinatorially for the classical root systems of types $A$, $B$, $C$, and $D$.
As a consequence we verify that this statistic coincides, in types $A$, $B$, and $D$, with the odd length statistics defined and studied in \cite{KV}, \cite{VS}, \cite{VS2}, \cite{BC2}, \cite{CTh}, and \cite{Lan} in 
these types. Our combinatorial computation of the statistic in the classical types shows that it is the sum of some more
fundamental statistics and we compute the signed (by length) multivariate generating function of these statistics in types $B$ and $D$. These results reduce to results in \cite{KV}, \cite{VS2}, and \cite{BC1} when all
the variables are equal.
We also show that the signed generating function of this statistic factors nicely for any crystallographic
root system except possibly in type $E_8$.

The organization of the paper is as follows. In the next section we recall some definitions, notation and
results that are used in the sequel. In \S 3 we define a new statistic on any Weyl group, which we call the odd
length. This statistic depends on the choice of a simple system in the root system of the Weyl group and we
show that its generating function over the Weyl group only depends on the root system. Using a convenient
choice of simple system we compute combinatorially the odd length of any element of any Weyl group of
classical type and verify that it coincides, in types $A$, $B$, and $D$, with the odd length statistics already
defined in \cite{KV}, \cite{VS}, \cite{VS2}, \cite{BC2}, and \cite{CTh}, in these types. 
In \S 4 we show that the signed generating function of the odd length over the symmetric group coincides with the one over the unimodal permutations. In \S 5 we compute, motivated by the results in \S 3 and using the
one proved in \S 4, the signed
multivariate distributions of certain statistics 
over the Weyl groups of types $B$ and $D$ and show that they factor nicely in almost all cases. Finally, in \S 6,
we show, using previous results and computer calculations, that the signed generating function of the odd length factors nicely for all irreducible crystallographic root systems except possibly in type $E_8$.

\section{Preliminaries}
In the following $V$ is a real vector space endowed with a symmetric bilinear form $(\cdot,\cdot)$. A reflection is a linear operator $s$ on $V$ which sends some nonzero vector $\alpha$ to its negative and fixes pointwise the hyperplane $H_\alpha$ orthogonal to it. For $v\in V$ the action of $s=s_\alpha$ is given by:
\[s_\alpha v=v-2\frac{(\alpha,v)}{(\alpha,\alpha)}v.\]
It is easy to see that $s_\alpha$ is an involution in $O(V)$, group of orthogonal transformations of $V$. Finite reflection groups are finite subgroups of $O(V)$ generated by reflections. We are interested in Coxeter groups of type $A$, $B$ and $D$, which arise as reflection groups of crystallographic root systems.

\begin{defn}
Let $V$, $(\cdot,\cdot)$ be as before. A finite subset $\Phi \subset V$ of nonzero vectors is a \emph{crystallographic root system} if it satisfies:
\begin{enumerate}
\item $\Phi \cap \mathbb R \Phi=\{\alpha,-\alpha\}$ for all $\alpha \in \Phi$
\item $s_{\alpha}\Phi=\Phi$ for all $\alpha \in \Phi$
\item $\frac{(\alpha,\beta)}{(\alpha,\alpha)} \in \mathbb Z$ for all $\alpha,\beta\in \Phi$.
\end{enumerate}
Vectors in $\Phi$ are called \emph{roots}.
\end{defn}
The group $W$ generated by the reflections $\{s_\alpha, \,\alpha \in \Phi\}$, is the \emph{Weyl group} of $\Phi$.

We call a subset $\Delta \subset \Phi$ a \emph{simple system} if it is a basis of the $\mathbb R-$span of  $\Phi$ in $V$ and if moreover each $\alpha \in \Phi$ is a linear combination of elements of $\Delta$ with all nonnegative or all nonpositive  coefficients. It turns out that simple systems exist (for details see \cite{Hum}) and that for crystallographic root systems all the roots are integer linear combinations of simple roots. The group $W$ is indeed generated by $S=\{s_\alpha,\,\alpha\in \Delta\}$, the set of simple reflections. Moreover $(W,S)$ is a Coxeter system. For an element $w=s_{\alpha_1}\cdots s_{\alpha_r}\in W$ and a root $\alpha$ we let $w(\alpha)$ denote the action of $w$ on $\alpha$ as composition of the reflections $s_{\alpha_1},\ldots,s_{\alpha_r}$.

For $\Delta=\{\alpha_s,\,s\in S\}$,  we let $\Phi_\Delta ^+$ denote the set of roots that are nonnegative linear combinations of simple roots, and $\Phi_\Delta ^-=-\Phi_\Delta^+$, so $\Phi_\Delta=\Phi_\Delta^+\cup \Phi_\Delta^-$.

\p For $\alpha \in \Phi$, $\alpha=\sum_{s\in S} c_s \alpha_s$, we call \emph{height} of $\alpha$ (with respect to $\Delta$) the sum of the coefficients of the linear combination: 
\begin{equation}\label{hg}
\hgt_\Delta(\alpha):=\sum_{s\in S}{c_s}.
\end{equation}

For a Coxeter system $(W,S)$ as above, the Coxeter length has an interpretation in terms of the action of $W$ on $\Phi$:\begin{equation}\label{lengthroot}
\ell(w)=|\{\alpha\in\Phi^+\,:\,w(\alpha)\in \Phi^-\}|,
\end{equation}
that is, for any element $w\in W$ it counts the number of positive roots sent to negative roots by its action as a composition of reflections.

Let now $\Phi$ be a crystallographic irreducible root system of type $A$, $B$, $C$, or $D$. We consider, in particular, for each of these types the following root systems:
\begin{enumerate}
\item $\Phi=\{\pm(e_i-e_j),\,1\leq i<j\leq n\}$ (type $A_{n-1}$),
\item $\Phi=\{\pm(e_i\pm e_j),\,1\leq i<j\leq n\}\cup\{e_i, i\in [n]\}$ (type $B_{n}$),
\item $\Phi=\{\pm(e_i\pm e_j),\,1\leq i<j\leq n\}\cup\{2e_i, i\in [n]\}$  (type $C_{n}$),
\item $\Phi=\{\pm(e_i\pm e_j),\,1\leq i<j\leq n\}$  (type $D_{n}$).
\end{enumerate}
For these systems, we will consider in the sequel the following convenient choices of simple systems:
\begin{enumerate}
\item $\Delta=\{(e_{i+1}-e_{i}),\, i\in [n-1]\}$, for type $A_{n-1}$,
\item $\Delta=\{(e_{i+1}- e_{i}),\, i\in [n-1]\}\cup\{e_1\}$, for type $B_{n}$,
\item $\Delta=\{(e_{i+1}- e_i),\, i\in [n-1]\}\cup\{2e_1\}$,  for type $C_{n}$,
\item $\Delta=\{(e_{i+1}- e_i),\, i\in [n-1]\} \cup \{ e_1 + e_2 \}$,  for type $D_{n}$.
\end{enumerate}
We recall here that for suitable sets of generators, the groups $W(\Phi)$ are not only Coxeter groups, but they have very nice combinatorial descriptions.  We employ here, for these groups, notation from in \cite[Chapter 8]{BB}. In particular, for the groups of (even) signed permutations we use the window notation.
\begin{pro}\label{prelim}
Let $\Delta\subseteq\Phi$ and  $S$ be as above. Then $(W(\Phi),S)$ is isomorphic to:
\begin{enumerate}
\item the symmetric group $S_n$, with Coxeter generators the simple transpositions $(i,i+1)$, for $i=1\ldots n-1$, if $\Phi$ is of type $A_{n-1}$;
\item the group of signed permutations $B_n$, with Coxeter generators the simple transpositions $(i,i+1)(-i,-i-1)$, for $i=1\ldots n-1$, and $s_0^B=[-1,2,\ldots,n]$, if $\Phi$ is of type $B_n$ or $C_n$;
\item the even hyperoctahedral group $D_n$, with Coxeter generators the simple transpositions $(i,i+1)(-i,-i-1)$, for $i=1\ldots n-1$, and $s_0^D=[-2,-1,3,\ldots,n]$, if $\Phi$ is of type $D_n$.
\end{enumerate}  
Moreover, with the above choices of simple systems, the Coxeter length has the following combinatorial interpretations in terms of statistics on the window notation
\[\ell_\Delta(\sigma)=\begin{cases} \inv(\sigma),& \mbox{ if $\Phi$ is of type $A$,}\\
 \inv(\sigma)+\nneg(\sigma)+\nsp(\sigma), &\mbox{ if $\Phi$ is of type $B$ or $C$,}\\
  \inv(\sigma)+\nsp(\sigma),& \mbox{ if $\Phi$ is of type $D$,}\end{cases}\]
  where $\inv(\sigma)=|\{(i,j)\in[n]^2\mid i<j,\,\sigma(i)>\sigma(j)\}|$, $\nsp(\sigma)=|\{(i,j)\in[n]^2\mid i<j,\,\sigma(i)+\sigma(j)<0\}|$, $\nneg(\sigma)=|\{i \in[n]\mid \sigma(i)<0\}|$.
\end{pro}

\section{Odd length}

In this section we define a new statistic on any Weyl group $W$ which we call the odd length.
While this statistics depends on the choice
of a simple system $\Delta \subseteq \Phi$, where $\Phi$ is the root system of $W$, 
we show that its generating function over the corresponding Weyl group does not. 
We then compute combinatorially this new statistics for the classical Weyl groups, for a natural
choice of simple system, and show that it coincides with the statistics by the same name that 
have already been defined and studied in \cite{BC1},   \cite{BC2}, \cite{KV}, 
\cite{VS}, and \cite{VS2}.

Let $\Phi$ be a root system and $W$ be the corresponding Weyl group. Let 
$\Delta \subseteq \Phi$ be a simple system for $\Phi$, and let 
$\Phi^{+}_{\Delta}$ and $\Phi^{-}_{\Delta}$ be the corresponding sets of positive and negative
roots. Given a positive root $\alpha \in \Phi^{+}_{\Delta}$ let $\hgt_{\Delta}(\alpha)$ be its height,
relative to $\Delta$.
For any $\sigma \in W$, we let
\begin{equation}
L_{\Phi(\Delta)}(\sigma):=|\{\alpha\in \Phi^{+}_{\Delta}\,: \, \hgt_{\Delta}(\alpha) \equiv 1 \pmod 2 ,\,
\sigma(\alpha)\in \Phi^{-}_{\Delta} \}|.
\end{equation}
We call  $L_{\Phi(\Delta)}(\sigma)$ the {\em odd length} of $\sigma$, and we call \emph{odd roots} the positive roots of odd height.

Our object of interest in this work is the signed (by length) generating function of the odd length over
the Weyl group. We now show that this generating function does not depend on the simple system
used to compute $L_{\Phi(\Delta)}$ and $\ell_{\Delta}$.

\begin{pro}
Let $\Phi$ be a root system and $\Delta, \Delta' \subseteq \Phi$ be simple systems for $\Phi$.
Then 
\[
\sum_{\sigma \in W} x^{\ell_{\Delta}(\sigma)} y^{L_{\Delta}(\sigma)} =
\sum_{\sigma \in W} x^{\ell_{\Delta'}(\sigma)} y^{L_{\Delta'}(\sigma)}.
\]
In particular, $\sum_{\sigma \in W} (-1)^{\ell_{\Delta}(\sigma)} y^{L_{\Delta}(\sigma)} =
\sum_{\sigma \in W} (-1)^{\ell_{\Delta'}(\sigma)} y^{L_{\Delta'}(\sigma)}$.
\end{pro}
\begin{proof}
It is well known (see, e.g., \cite[\S 1.4]{Hum}) that in our hypotheses there is a $w \in W$ 
such that $w(\Delta) = \Delta'$, and so $w(\Phi^{+}_{\Delta})= \Phi^{+}_{\Delta'}$. From this it
follows easily that for any $\alpha \in \Phi^{+}_{\Delta}$ we have that $\hgt_{\Delta}(\alpha) =
\hgt_{\Delta'}(w(\alpha))$. Therefore 
\begin{eqnarray*}
L_{\Delta'}(\tau)&=& 
|\{\alpha\in \Phi^{+}_{\Delta'}\,: \, \hgt_{\Delta'}(\alpha) \equiv 1 \pmod 2 ,\,
\tau(\alpha)\in \Phi^{-}_{\Delta'} \}| \\
&=&
|\{\beta \in \Phi^{+}_{\Delta}\,: \, \hgt_{\Delta'}(w(\beta)) \equiv 1 \pmod 2 ,\,
\tau(w(\beta))\in \Phi^{-}_{\Delta'} \}| \\
&=& |\{\beta\in \Phi^{+}_{\Delta}\,: \, \hgt_{\Delta}(\beta) \equiv 1 \pmod 2 ,\,
\tau(w(\beta))\in w(\Phi^{-}_{\Delta}) \}| \\
&=&L_{\Delta}(w^{-1} \tau w)
\end{eqnarray*}
and similarly $\ell_{\Delta'}(\tau) = \ell_{\Delta}(w^{-1} \tau w)$, for all $\tau \in W$. 
Hence
\[
\sum_{\tau \in W} x^{\ell_{\Delta'}(\tau)} y^{L_{\Delta'}(\tau)}=
\sum_{\tau \in W} x^{\ell_{\Delta}(w^{-1} \tau w)} y^{L_{\Delta}(w^{-1} \tau w)} =
\sum_{\sigma \in W} x^{\ell_{\Delta}(\sigma)} y^{L_{\Delta}(\sigma)}.
\]
\end{proof}

We now derive combinatorial descriptions of $L_{\Phi(\Delta)}$ for the root systems of types 
$A$,  $B$, $C$, and $D$, for the same choice of simple systems used in Proposition \ref{prelim}.
For $\sigma \in B_n$ we let, following \cite{BC1} (see also \cite{CTh})
\begin{eqnarray*}
\oneg(\sigma)&:=&|\{i\in [n] \,|\,\sigma(i)<0, \; i \equiv 1 \pmod 2 \}|\\
\onsp(\sigma)&:=&|\{(i,j)\in [n]\times[n] \,|\, i<j, \,\sigma(i)+\sigma(j)<0, \; j-i \equiv 1 \pmod 2 \}|, 
\end{eqnarray*}
and
\[
\oinv(\sigma) := |\{(i,j)\in [n]\times[n] \,|\, i<j, \,\sigma(i) > \sigma(j), \; j-i \equiv 1 \pmod 2\}|,
\]
and define similarly $\eneg$, $\ensp$, and $\einv$.

Let $\Phi$ be a cristallographic root system of type $A$, $B$, $C$, or $D$, and
$\Delta \subseteq \Phi$ be the simple system considered in Proposition \ref{prelim}.

%
\begin{pro} \label{oddcomb}
Let $\Phi$ be a cristallographic root system of type $A$, $B$, $C$, or $D$, and 
$\Delta \subseteq \Phi$ be as above. Then
\[
L_{\Phi(\Delta)}(\sigma)=\begin{cases} 
\oinv(\sigma), \quad\mbox{if $\Phi$ is of type $A$} \\
\oneg(\sigma)+\oinv(\sigma)+\onsp(\sigma), \quad\mbox{if $\Phi$ is of type $B$}, \\
\nneg(\sigma)+\oinv(\sigma)+\ensp(\sigma), \quad\mbox{if $\Phi$ is of type $C$}, \\
\oinv(\sigma)+\onsp(\sigma), \quad\mbox{if $\Phi$ is of type $D$}, \\
\end{cases}
\]
for all $\sigma$ in the Weyl group of $\Phi$.							
\end{pro}
\begin{proof}
In all types a simple computation shows that
$
\hgt_\Delta(-e_i+e_j)=j-i,
$
for all $1\leq i<j\leq n$. Furthermore, if $\Phi$ is of type $B$ then one obtains that
$
\hgt_\Delta(e_i)=i $ for all $1 \leq i \leq n$ and $\hgt_\Delta(e_i+e_j)=i+j$ for
all $1\leq i<j\leq n$
while
$\hgt_\Delta(e_i+e_j)=i+j-1$ for all $1\leq i \leq j\leq n$
if $\Phi$ is of type $C$, and
$ \hgt_\Delta(e_i+e_j)=i+j-2$
if $\Phi$ is of type $D$, for all $1\leq i<j\leq n$.
\end{proof}

So, for example, if  $n=5$, and $\sigma=[3,-1,-4,-2,5]$, then $L_{\Phi(\Delta)}(\sigma)=6$, if $\Phi$ is of type $B$,
while $L_{\Phi(\Delta)}(\sigma)=8$ if $\Phi$ is of type $C$.

Proposition \ref{oddcomb} shows that in types $A$, $B$, and $D$, with the choice of simple system
made there, $L_{\Phi(\Delta)}$ coincides with the odd length $L$ defined and studied in \cite{BC1},  \cite{BC2}, \cite{KV}, 
\cite{VS}, and \cite{VS2}.

\section{Type $A$}
We showed in the previous section that the odd length defined combinatorially for type $A$ coincides with $L_{\Phi(\Delta)}$ for a very natural choice of simple system of type $A_{n-1}$.  Nice formulae for the signed (by length) distribution of this statistic over all quotients of the symmetric groups were proved in \cite{BC1}. For later use (see Section \ref{sec:multiD}), we prove here that the signed generating function of $L_{\Phi(\Delta)}(=\oinv)$ over $S_n$ is the same as the one over the set of unimodal permutations, whose definition we now recall. 
\begin{defn}
Let $\sigma \in S_n$. We say that $i\in [2,n-1]$ is a \emph{peak} if $\sigma(i-1)<\sigma(i)>\sigma(i+1)$.
\end{defn}
\begin{defn}
We say that a permutation $\sigma\in S_n$ is \emph{unimodal} if it has no peaks. We denote by $U_n$ the set of  unimodal permutations.
\end{defn}

\begin{lem}\label{lem:unim} Let $n\in \N$. Then
\[\sum_{\sigma \in S_n}(-1)^{\ell(\sigma)}x^{\oinv(\sigma)}=\sum_{\sigma \in U_n}(-1)^{\ell(\sigma)}x^{\oinv(\sigma)}.\]
\end{lem}
\begin{proof}
Let $\sigma \in S_n \setminus U_n$. Let $R_{\sigma}:=\{\sigma(i)\mid i \mbox{ peak}\}$ be the set of  the values of the images of the peaks of $\sigma$. By hypothesis $R_{\sigma}$ is non-empty. Let $r$ be  such that $\sigma(r)=\max R_{\sigma}$ and define the involution $\sigma^r:=\sigma(r-1,r+1)$. Then $\ell(\sigma^r)=\ell(\sigma)\pm1$ while $\oinv(\sigma^r)=\oinv(\sigma)$. Thus
\[\sum_{\sigma \in S_n}(-1)^{\ell(\sigma)}x^{\oinv(\sigma)}=\sum_{\sigma \in U_n}(-1)^{\ell(\sigma)}x^{\oinv(\sigma)},\] as desired.\end{proof}

In fact, a finer result holds: the signed generating function is the same when restricted to chessboard elements, defined in \cite{VS}, whose definition we recall here.
\begin{defn}
	We say that a (even signed)  permutation $\sigma\in S_n$ (or  $D_n$) is \emph{chessboard} if $\sigma(i)\equiv i$ for all $i\in[n]$ or if  $\sigma(i)\equiv i+1 $ for all $i \in [n ]$. We write $C(S_n)$, resp.\  $C(D_n)$,  for the subgroup of the chessboard elements of the relative group and for $X\subset S_n$ (or $D_n$) we denote $C(X)=X\cap C(S_n)$, resp.\ $C(X)=X\cap C(D_n)$.
\end{defn}

\begin{rmk}
As the involution defined  in the proof of Lemma~\ref{lem:unim} preserves the parity of the entries in all positions, the same equality holds true when restricting the supports of the sums on both sides to chessboard elements, 
\[\sum_{\sigma \in C(S_n)}(-1)^{\ell(\sigma)}x^{\oinv(\sigma)}=\sum_{\sigma \in C(U_n)}(-1)^{\ell(\sigma)}x^{\oinv(\sigma)}.\qedhere\]  \end{rmk}
\section{Signed multivariate distributions}
Taking inspiration from the combinatorial descriptions of the odd length defined in the previous section, we define here some natural generalizations of these statistics and study their signed (multivariate) distributions over the classical Weyl groups. In all cases, we show that these generating functions factor in a very explicit way.

\subsection{Type $B$}
In this section we study the signed multivariate distributions of
the statistics $\oneg$, $\eneg$, $\oinv$, $\onsp$, and $\ensp$ over
the classical Weyl groups of type $B_n$. In almost all cases we show that these factor in a very nice way. In particular,
we obtain the signed generating function of $L_{\Phi}$ for
root systems of types ($B$ and)~$C$.

We begin with the following lemmas. Their proofs are straightforward  verifications from the definitions and are 
therefore omitted.

\begin{lem}\label{tilde}
Let $n\in \N$, $n\geq 2$, $\sigma \in B_{n-1}$. Let $\tilde\sigma:=[\sigma(1),\ldots,\sigma(n-1),-n]$  and 

\p $\delta$:=$\chi( n \equiv 0 \pmod 2 )$.
Then:
\vspace{4mm}

\begin{tabular}{ll}
$\oneg(\tilde\sigma)=\oneg(\sigma)+1-\delta$ & $\eneg(\tilde\sigma)=\eneg(\sigma)+\delta$ \\
$\oinv(\tilde\sigma)=\oinv(\sigma)+\left\lceil\frac{n-1}{2}\right\rceil$ & $\einv(\tilde\sigma)=\einv(\sigma)+\left\lfloor\frac{n-1}{2}\right\rfloor$\\
$\onsp(\tilde\sigma)=\onsp(\sigma)+\left\lceil\frac{n-1}{2}\right\rceil$ & $\ensp(\tilde\sigma)=\ensp(\sigma)+\left\lfloor\frac{n-1}{2}\right\rfloor$ .\\
\end{tabular}
\end{lem}

\begin{lem}\label{hat}
Let $n\in \N,$ $n\geq 2$, $\sigma\in B_{n-1}$. Let $\hat\sigma:=[n,\sigma(1),\ldots,\sigma(n-1)]$. Then:

\vspace{4mm}
\begin{tabular}{ll}
$\oneg(\hat\sigma)=\eneg(\sigma)$ & $\eneg(\hat\sigma)=\oneg(\sigma)$ \\
$\oinv(\hat\sigma)=\oinv(\sigma)+\left\lceil\frac{n-1}{2}\right\rceil$ & $\einv(\hat\sigma)=\einv(\sigma)+\left\lfloor\frac{n-1}{2}\right\rfloor$ \\
$\onsp(\hat\sigma)=\onsp(\sigma)$ & $\ensp(\hat\sigma)=\ensp(\sigma)$. \\
\end{tabular}
\end{lem}

\begin{lem}\label{check}
Let $n\in \PP,$ $n\geq 2$, $\sigma\in B_{n-1}$, and $\check\sigma:=[-n,\sigma(1),\ldots,\sigma(n-1)]$. Then:
\vspace{4mm}

\begin{tabular}{ll}
$\oneg(\check\sigma)=\eneg(\sigma)+1$ & $\eneg(\check\sigma)=\oneg(\sigma)$ \\
$\oinv(\check\sigma)=\oinv(\sigma)$ & $\einv(\check\sigma)=\einv(\sigma)$ \\
$\onsp(\check\sigma)=\onsp(\sigma)+\left\lceil\frac{n-1}{2}\right\rceil$ & $\ensp(\check\sigma)=\ensp(\sigma)+\left\lfloor\frac{n-1}{2}\right\rfloor$ .\\
\end{tabular}
\end{lem}

The key observation to prove the formulae of the signed multivariate distributions is  the following, which is analogous to \cite[Lemma~3.3]{BC2}.

\begin{lem}\label{zerotri}
Let $\sigma \in B_n$, $s \in \{ \oneg, \eneg, \onsp, \ensp \}$, and $a \in [\pm n] \setminus \{\pm 1,\pm n\}$. Then, if $\sigma^*:=\sigma(a-1,a+1)(-a+1,-a-1)$, one has:
\[ s(\sigma^*)=s(\sigma),  \quad \quad
\ell(\sigma^*)=\ell(\sigma)\pm 1.
\]
Furthermore, if $a=\sigma^{-1}(n)$, then $\oinv(\sigma^*)=\oinv(\sigma)$.
\end{lem}

In the following we let $x_1^{o(\sigma)}$ denote $x_1^{\oneg(\sigma)}$, $x_2^{e(\sigma)}$ denote $x_2^{\eneg(\sigma)}$, $y^{o(\sigma)}$ denote $y^{\oinv(\sigma)}$, $z^{o(\sigma)}$ denote $z^{\onsp(\sigma)}$, and analogously for the even statistics, to lighten the notation.

\begin{thm}
\label{triv2}
Let $n \in \mathbb P$.  Then 
\begin{equation*}
\sum_{\sigma \in B_{n}}{(-1)^{\ell(\sigma)}x_1^{o(\sigma)} x_2^{e(\sigma)} y^{o(\sigma)}z^{e(\sigma)}}  = 
\left\{ \begin{array}{ll}
\displaystyle 
\prod_{i=1}^{n-1}(1+ (-1)^i y^{\lceil \frac{i}{2} \rceil})
\prod_{i=0}^{\lfloor \frac{n-2}{2} \rfloor}(1-x_1 x_2 z^{2i}), 
& \mbox{if $n \equiv 0 \pmod 2$,} \\
\displaystyle (1-x_1 z^{\frac{n-1}{2}})
\prod_{i=1}^{n-1}(1+ (-1)^i y^{\lceil \frac{i}{2} \rceil})
\prod_{i=0}^{\lfloor \frac{n-2}{2} \rfloor}(1-x_1 x_2 z^{2i}),  
& \mbox{if $n \equiv 1 \pmod 2$,}
\end{array} \right.
\end{equation*}

\end{thm}
\begin{proof}
We proceed by induction on $n \in \PP$, the result being easy to check if $n \leq 2$. Assume $n \geq 3$. Let $f(\sigma)=(-1)^{\ell(\sigma)}x_1^{o(\sigma)} x_2^{e(\sigma)} y^{o(\sigma)}z^{e(\sigma)}$ for all $\sigma \in B_n$. 
By Lemma \ref{zerotri} we have:
\begin{eqnarray}\nonumber
\sum_{\sigma \in B_n}{f(\sigma)}&=& \sum_{ \substack{\{\sigma\in B_n :\\ \sigma(n)=n\}}}{f(\sigma)} +\sum_{ \substack{\{\sigma\in B_n :\\ \sigma(-n)=n\}}}{f(\sigma)} +\sum_{ \substack{\{\sigma\in B_n :\\ \sigma(1)=n\}}}{f(\sigma)}+\sum_{\substack{\{\sigma\in B_n :\\ \sigma(-1)=n\}}}{f(\sigma)}\\\label{sumtri}
&=& \sum_{\sigma \in B_{n-1}}
\big( f(\sigma)+f(\tilde\sigma)+f(\hat\sigma)+f(\check\sigma) \big) 
\end{eqnarray}
where $\tilde\sigma,\,\hat\sigma,\,\check\sigma$ are as in the previous lemmas.
But, by Lemmas \ref{tilde}, \ref{hat}, and \ref{check} we have that
\[
(-1)^{\ell(\tilde\sigma)}x_1^{o(\tilde\sigma)} x_2^{e(\tilde\sigma)} y^{o(\tilde\sigma)}z^{e(\tilde\sigma)}
=x_1^{1-\delta(n)} x_2^{\delta(n)} y^{\lceil \frac{n-1}{2} \rceil } z^{\lfloor \frac{n-1}{2} \rfloor} (-1)^{n-1} (-1)^{\ell(\sigma)} x_1^{o(\sigma)} x_2^{e(\sigma)} y^{o(\sigma)} z^{e(\sigma)},
\] 
\[
{(-1)^{\ell(\hat\sigma)}x_1^{o(\hat\sigma)}x_2^{e(\hat\sigma)}
y^{o(\hat\sigma)}z^{e(\hat\sigma)}}= (-1)^{n-1} y^{\lceil \frac{n-1}{2}  \rceil} {(-1)^{\ell(\sigma)}x_1^{e(\sigma)}x_2^{o(\sigma)}
y^{o(\sigma)}z^{e(\sigma)}} \\
\]
and
\[
{(-1)^{\ell(\check\sigma)}x_1^{o(\check\sigma)}x_2^{e(\check\sigma)}
y^{o(\check\sigma)} z^{e(\check\sigma)}}= x_1 z^{\lfloor \frac{n-1}{2} \rfloor} 
{(-1)^{\ell(\sigma)}x_1^{e(\sigma)} x_2^{o(\sigma)} y^{o(\sigma)} z^{e(\sigma)}}
\]
for all $\sigma \in B_{n-1}$. 

Suppose now that $n \equiv 0 \pmod 2$. Then by our induction hypotheses we have that
\begin{eqnarray*}
\sum_{\sigma \in B_{n}} f(\sigma)
& = & (1-x_2 y^{\frac{n}{2}} z^{\frac{n-2}{2}}) 
\sum_{\sigma\in B_{n-1}} (-1)^{\ell(\sigma)} x_1^{o(\sigma)} x_2^{e(\sigma)} y^{o(\sigma)} z^{e(\sigma)} \\
&+& (x_1 z^{\frac{n-2}{2}}-y^{\frac{n}{2}}) 
\sum_{\sigma\in B_{n-1}} (-1)^{\ell(\sigma)} x_1^{e(\sigma)}  x_2^{o(\sigma)} y^{o(\sigma)} z^{e(\sigma)} \\
&=&(1-x_2 y^{\frac{n}{2}} z^{\frac{n-2}{2}}) 
(1-x_1 z^{\frac{n-2}{2}}) 
\prod_{i=1}^{n-2}(1+(-1)^{i} y^{\lceil  \frac{i}{2} \rceil}) \prod_{i=0}^{\frac{n-4}{2}}(1-x_1 x_2 z^{2i}) \\
& + & (x_1 z^{\frac{n-2}{2}}-y^{\frac{n}{2}})(1-x_2 z^{\frac{n-2}{2}}) 
\prod_{i=1}^{n-2} (1+(-1)^{i} y^{\lceil  \frac{i}{2} \rceil}) \prod_{i=0}^{\frac{n-4}{2}} (1-x_1 x_2 z^{2i}) \\
 & = & (1+x_1 x_2 y^{\frac{n}{2}} z^{n-2}-y^{\frac{n}{2}}-x_1 x_2 z^{n-2}) 
\prod_{i=1}^{n-2} (1+(-1)^{i} y^{\lceil  \frac{i}{2} \rceil}) \prod_{i=0}^{\frac{n-4}{2}}(1-x_1 x_2 z^{2i})
\end{eqnarray*}
and the result follows. Similarly, if $n \equiv 1 \pmod 2$ 
then we obtain that
\begin{eqnarray*}
\sum_{\sigma \in B_{n}} f(\sigma)
& = & (1+x_1 y^{\frac{n-1}{2}} z^{\frac{n-1}{2}}) 
\sum_{\sigma\in B_{n-1}} (-1)^{\ell(\sigma)} x_1^{o(\sigma)} x_2^{e(\sigma)} y^{o(\sigma)} z^{e(\sigma)} \\
&+& (x_1 z^{\frac{n-1}{2}}+y^{\frac{n-1}{2}}) 
\sum_{\sigma\in B_{n-1}} (-1)^{\ell(\sigma)} x_1^{e(\sigma)}  x_2^{o(\sigma)} y^{o(\sigma)} z^{e(\sigma)} \\
&=&(1+x_1 y^{\frac{n-1}{2}} z^{\frac{n-1}{2}}) 
\prod_{i=1}^{n-2}(1+(-1)^{i} y^{\lceil  \frac{i}{2} \rceil}) \prod_{i=0}^{\frac{n-3}{2}}(1-x_1 x_2 z^{2i}) \\
& + &(x_1 z^{\frac{n-1}{2}}+y^{\frac{n-1}{2}}) 
\prod_{i=1}^{n-2} (1+(-1)^{i} y^{\lceil  \frac{i}{2} \rceil}) \prod_{i=0}^{\frac{n-3}{2}} (1-x_1 x_2 z^{2i}) \\
 & = & (1+x_1 y^{\frac{n-1}{2}} z^{\frac{n-1}{2}}+y^{\frac{n-1}{2}}+x_1  z^{\frac{n-1}{2}}) 
\prod_{i=1}^{n-2} (1+(-1)^{i} y^{\lceil  \frac{i}{2} \rceil}) \prod_{i=0}^{\lfloor \frac{n-2}{2} \rfloor}(1-x_1 x_2 z^{2i}) \\
 & = & (1+x_1 z^{\frac{n-1}{2}}) (1+y^{\frac{n-1}{2}}) 
\prod_{i=1}^{n-2} (1+(-1)^{i} y^{\lceil  \frac{i}{2} \rceil}) \prod_{i=0}^{\lfloor \frac{n-2}{2} \rfloor}(1-x_1 x_2 z^{2i}),
\end{eqnarray*}
and the result again follows.
\end{proof}

As an immediate corollary of the previous result we obtain the generating function for the signed distribution of $L_{\Phi}$
for root systems of type $C$.
\begin{cor}
Let $n\in \N$, $n\geq 2$. Then
\[
\sum_{\sigma \in W(\Phi(C_n))}{(-1)^{\ell(\sigma)}x^{L_{\phi(C_n)}(\sigma)}}= (1-x^{\left\lceil\frac{n}{2}\right\rceil})\prod_{i=1}^{\left\lceil\frac{n}{2}\right\rceil} (1-x^{2i})^2.
\]\end{cor}
\begin{proof}
This follows immediately from Theorem \ref{triv2} by letting  $x_1=x_2=y=z=x$.
\end{proof}

The following result gives the signed  generating function of $x^{\oneg(\cdot)}y^{\oinv(\cdot)}
z^{\onsp(\cdot)}$ and $x^{\eneg(\cdot)}y^{\oinv(\cdot)}
z^{\onsp(\cdot)}$ over the hyperoctahedral group $B_n$.
The proof is analogous to that of Theorem \ref{triv2} and is therefore omitted.

\begin{thm}\label{triv1}
Let $n \in \mathbb P$.  Then 
\begin{equation}
\label{triooo}
\sum_{\sigma \in B_{n}}{(-1)^{\ell(\sigma)}x^{o(\sigma)}y^{o(\sigma)}z^{o(\sigma)}}  = \left\{ \begin{array}{ll}
\displaystyle (1-x) (1-y^\frac{n}{2} z^\frac{n}{2})
\prod_{i=1}^{\lfloor \frac{n-1}{2} \rfloor}(1-x z^{2i})(1-y^{2i}), 
& \mbox{if $n \equiv 0 \pmod 2$,} \\
\displaystyle (1-x) 
\prod_{i=1}^{\lfloor \frac{n-1}{2} \rfloor}(1-x z^{2i})(1-y^{2i}), 
& \mbox{if $n \equiv 1 \pmod 2$,}
\end{array} \right.
\end{equation}
and
\begin{equation}
\label{trieoo}
\sum_{\sigma \in B_{n}}{(-1)^{\ell(\sigma)}x^{e(\sigma)}y^{o(\sigma)}z^{o(\sigma)}}  = \left\{ \begin{array}{ll}
\displaystyle (1-x) (z^\frac{n}{2}-y^\frac{n}{2} )
\prod_{i=1}^{\lfloor \frac{n-1}{2} \rfloor}(1-x z^{2i})(1-y^{2i}), 
& \mbox{if $n \equiv 0 \pmod 2$,} \\
\displaystyle 0, 
& \mbox{if $n \equiv 1 \pmod 2$.}
 \end{array} \right.
\end{equation}

\end{thm}

Note that the generating function of $(-1)^{\ell(\sigma)} x_{1}^{\oneg(\sigma)} x_{2}^{\eneg(\sigma)}y^{\oinv(\sigma)}z^{\onsp(\sigma)}$
over $\sigma \in B_n$
does not factor nicely in general. For example, 
$ \sum_{\sigma \in B_4} (-1)^{\ell(\sigma)} x_{1}^{\oneg(\sigma)} x_{2}^{\eneg(\sigma)}y^{\oinv(\sigma)}z^{\onsp(\sigma)} =
(1-y^2) (1- x_1 x_2 z^2) (1+x_1 x_2 y^2 z^2 -x_1 x_2 z^2 
- x_2 y^2 z^2 + x_1 z^2 + x_2 y^2 -x_1 -y^2 )$.

As an immediate corollary of Theorem \ref{triv1} we obtain the
following result, which also follows from Proposition \ref{oddcomb} and
Theorem 5.4 of \cite{BC1}.
\begin{cor}
Let $n\in \N$, $n\geq 2$. Then
\[
\sum_{\sigma \in W(\Phi(B_n))}{(-1)^{\ell(\sigma)}x^{L_{\Phi(B_n)}(\sigma)}}= \prod_{i=1}^{n} (1-x^{i}).
\]
\end{cor}

We conclude by noting the following univariate natural special
cases of the multivariate results in this section.
For $n\in \mathbb P$, and $\sigma \in B_n$, we let
\begin{eqnarray*}
L_{ooe}(\sigma)&:=&\oneg(\sigma)+\oinv(\sigma)+\ensp(\sigma) \\
L_{eoe}(\sigma)&:=&\eneg(\sigma)+\oinv(\sigma)+\ensp(\sigma) \\
L_{eoo}(\sigma)&:=&\eneg(\sigma)+\oinv(\sigma)+\onsp(\sigma).
\end{eqnarray*}

\begin{cor}\label{cor:uni}
Let $n \in \N$, $n\geq 3$. Then
\begin{equation}\label{ooe}
\sum_{\sigma \in B_n}{(-1)^{\ell(\sigma)}x^{L_{ooe}(\sigma)}}=(1-x^{\left\lceil \frac{n}{2}\right\rceil})\prod_{i=i}^{n-1} (1-x^i),
\end{equation}
\begin{equation}\label{eoe}
\sum_{\sigma \in B_n}{(-1)^{\ell(\sigma)}x^{L_{eoe}(\sigma)}}=(1-x^{\left\lfloor \frac{n}{2}\right\rfloor})\prod_{i=i}^{n-1} (1-x^i ),
\end{equation}
and
\begin{equation}\label{eoo}
\sum_{\sigma \in B_n}{(-1)^{\ell(\sigma)}x^{L_{eoo}(\sigma)}}=0.
\end{equation}
\end{cor}

\subsection{Type $D$}\label{sec:multiD}
As for type $B$, we derive in this section signed multivariate generating functions for the statistics $\oinv,\, \onsp,\,\ensp$ over the even hyperoctahedral groups.

Our first result shows that 
 the signed joint distribution of $\oinv$ and $\ensp$ is zero.
\begin{pro}
	Let $n\in \mathbb P$. Then
	\begin{equation}
	\sum_{\sigma\in D_n}{(-1)^{\ell_D (\sigma)}x^{\oinv (\sigma)}y^{\ensp(\sigma) }} =0
	\end{equation}
\end{pro}
\begin{proof}
We define, for $\sigma \in D_n$, the following involution:
$$\overline \sigma=\begin{cases} \sigma s_1, & \mbox{ if } ||\sigma^{-1}(1)|-|\sigma^{-1}(2)||\equiv 2 \pmod 2  \\
  \sigma s_0^D, & \mbox{ if } ||\sigma^{-1}(1)|-|\sigma^{-1}(2)||\equiv 1 \pmod 2. \end{cases}$$
  It is clear that in both cases $\ell_D(\overline{\sigma})=\ell_D(\sigma)\pm 1$. We now show that, for all $\sigma\in D_n$,  $\ensp(\sigma)=\ensp(\overline{\sigma})$ and $\oinv(\sigma)=\oinv(\overline{\sigma})$.
  
  Consider $\sigma$ for which the entries of absolute values $1$ and $2$ appear at an even distance. In this case the involution is defined by right multiplication by $s_1$, that is it exchanges these values. As the involution involves no sign changes,  $\ensp(\sigma)=\ensp(\overline{\sigma})$. The only inversion involved is between positions at even distance, thus  $\oinv(\sigma)=\oinv(\overline{\sigma})$. Similar reasoning shows that these equalities hold also in the other case. This implies the result.
\end{proof}
Similarly to Corollary~\ref{cor:uni}, the previous result implies the following result about the univariate signed generating function of  the statistic
\[L_{oe}(\sigma):=\oinv(\sigma)+\ensp(\sigma).\]
\begin{cor}
		Let $n\geq 2$. Then
		\begin{equation}
		\sum_{\sigma\in D_n}{(-1)^{\ell_D (\sigma)}x^{L_{oe}}(\sigma) } = 
		0
			\end{equation}
\end{cor}
We now study the signed bivariate generating function that refines the one of the odd length $L_{\Phi(D_n)}$: $\sum_{w\in D_n}(-1)^{\ell_D(w)}x^{\oinv(w)}y^{\onsp(w)}$. 
We will need some preliminary results. 

The next lemma shows that, as in the case of the symmetric and hyperoctahedral groups, the signed generating function of the odd length over $D_n$ is the same when restricted to chessboard elements. We prove a finer result, namely that this holds also for the signed bivariate generating function of odd inversions and odd negative sum pairs.
\begin{lem}\label{lem:chb}
Let $n \geq 2$. Then 
\[\sum_{\sigma \in D_n}{(-1)^{\ell_D(\sigma)}x^{\oinv(\sigma)}y^{\oneg(\sigma)}}=\sum_{\sigma \in C(D_n)}{(-1)^{\ell_D(\sigma)}x^{\oinv(\sigma)}y^{\oneg(\sigma)}}.\]
\end{lem}
\begin{proof}Let $\sigma \in D_n \setminus C(D_{n})$. Then there exists $i\in[n-1]$ such that $\sigma^{-1}(i)\equiv \sigma^{-1}(i+1) \pmod 2$. 
Let $i$ be minimal with this property and define $\sigma^{*}=s_{i} \sigma$. It is a well defined involution on $D_n \setminus C(D_{n})$. Clearly $\oinv(\sigma^{*})=\oinv(\sigma)$ and $\onsp(\sigma^{*})=\onsp(\sigma)$, while $\ell_D(\sigma^{*})=\ell_D(\sigma) \pm 1$. This implies the thesis.
\end{proof}
Recall that for $\sigma\in D_n$ the descent set is
$\Des(\sigma)=\{i\in [0,n-1]\mid \sigma(i)>\sigma(i+1)\}$, where we set $\sigma(0):=-\sigma(2)$. Also, recall that $S_n$ is naturally isomorphic to the  parabolic subgroup ${D_n}_{[n-1]}$ of $D_n$, and that $D_n$ can be written as 
\[D_n=T_n S_n,\]
where $T_n=\{\tau \in D_n \mid \Des(\tau)\subseteq \{0\}\}$. That is, every even signed permutation $\sigma\in D_n$ can be uniquely written as $\sigma=\sigma^{[n-1]}\sigma_{[n-1]}$, with  $\sigma^{[n-1]}\in T_n$ and $\sigma_{[n-1]}\in S_n$. Moreover,
 \[\ell_D(\sigma)=\ell_D(\sigma^{[n-1]})+\ell_D(\sigma_{[n-1]}),\]
 we refer the reader to \cite[Chapter 8.2]{BB} for further details. This last property does not hold in general for $L_{\Phi(D_n)}$. It does, however, for $L_{\Phi(D_n)}$ on  a special subset of chessboard elements, which we now define.

\begin{defn}
We say that an even signed permutation $\sigma$ is a \emph{good} chessboard element if $\sigma, \sigma^{[n-1]}$ and $\sigma_{[n-1]}$  are chessboard elements. We write $gC(D_n)$ for good chessboard elements of $D_n$.\end{defn}
In the following lemma we show that the odd inversions and the odd negative sum pairs (and thus the odd length $L_{\Phi(D_n)}$) are additive with respect to the parabolic factorisation $D_n=T_n S_n$ on good chessboard elements. 
\begin{lem}\label{lem:add}
 Let $\sigma \in gC(D_n)$. Then 
 $$\oinv(\sigma)=\oinv(\sigma^{[n-1]})+\oinv(\sigma_{[n-1]}) \quad \mbox{and} \quad \onsp(\sigma)=\onsp(\sigma^{[n-1]})+\onsp(\sigma_{[n-1]}), $$
 where $\sigma=\sigma^{[n-1]}\sigma_{[n-1]}$, $\sigma^{[n-1]}\in C(T_n)$ and $\sigma_{[n-1]}\in C(S_n)$.
\end{lem}
\begin{proof}
Let $\sigma$ be a good chessboard element. The set of inversions of $\sigma$ and $\sigma_{[n-1]}$ coincide. Moreover it is clear that $\oinv(\sigma^{[n-1]})=0$, thus
\[\oinv(\sigma)=\oinv(\sigma_{[n-1]})=\oinv(\sigma^{[n-1]})+\oinv(\sigma_{[n-1]}).\]
Since by assumption $\sigma^{[n-1]}$ is a chessboard element, the relative parities of pairs with negative sum are the same as for $\sigma$. Clearly $\onsp(\sigma_{[n-1]})=0$, thus
\[\onsp(\sigma)=\onsp(\sigma^{[n-1]})=\onsp(\sigma^{[n-1]})+\onsp(\sigma_{[n-1]}).\qedhere\]
\end{proof}
We show now that the signed bivariate generating function equals the one over good chessboard elements. The result follows from its analogue for type $B$.
\begin{lem}\label{lem:good}
	Let $n\geq 2$. Then
\[\sum_{\sigma\in D_n}{(-1)^{\ell_D (\sigma)}x^{\oinv (\sigma) }y^{\onsp (\sigma)}}=\sum_{\sigma\in gC(D_n)}{(-1)^{\ell_D (\sigma)}x^{\oinv (\sigma) }y^{\onsp (\sigma)}}.\]
\end{lem}
\begin{proof}
The lemma follows by \cite[Lemma~16 and Lemma~19]{VS}, observing that the involution defined in \cite[Lemma~19]{VS} restricts to an involution on the relevant subset of $D_n$, since it does not involve sign changes.
\end{proof}
The next theorem  implies (and gives a direct proof of) \cite[Corollary~4.2]{BC2}.
\begin{thm}
	Let $n\geq 2$. Then
\[\sum_{\sigma\in D_n}{(-1)^{\ell_D (\sigma)}x^{\oinv (\sigma) }y^{\onsp (\sigma)}}= \left(\sum_{\sigma\in S_n }{(-1)^{\ell (\sigma)}x^{\oinv (\sigma)}} \right)\left(\sum_{\sigma\in S_n }{(-1)^{\ell (\sigma)}y^{\oinv (\sigma)}}\right).\]
\end{thm}
\begin{proof}Thanks to Lemma~\ref{lem:good} and \ref{lem:add} the sum on the left hand side can be rewritten as
\begin{align}\nonumber
&\sum_{\sigma\in D_n}{(-1)^{\ell_D (\sigma)}x^{\oinv (\sigma) }y^{\onsp (\sigma)}}=\sum_{\sigma\in gC(D_n)}{(-1)^{\ell_D (\sigma)}x^{\oinv (\sigma) }y^{\onsp (\sigma)}}=\\ \label{eq:fact} &\left(\sum_{\sigma\in C(T_n)}{(-1)^{\ell_D (\sigma)}y^{\onsp (\sigma)}}\right)\left(\sum_{\sigma\in C(S_n)}{(-1)^{\ell (\sigma)}x^{\oinv (\sigma) }}\right).
\end{align}
We claim that the first factor of \eqref{eq:fact} equals the signed distribution of the odd length on the symmetric group. Consider the map $$|\cdot|:D_n \rightarrow S_n,\qquad\sigma = [\sigma(1),\ldots,\sigma(n)] \mapsto |\sigma|=[|\sigma(1)|,\ldots,|\sigma(n)|].$$ Its restriction to $C(T_n)$ is a bijection onto $C(U_n)$, the set of chessboard unimodal permutations. It is easy to see that (odd) negative sum pairs of elements of $C(T_n)$ are (odd) inversions of their images in  $C(U_n)$ through $|\cdot|$. More precisely, for $\sigma \in C(T_n)$
$$\nsp(\sigma)=\inv(|\sigma|), \qquad\onsp(\sigma)=\oinv(|\sigma|).$$
This observation proves that indeed
\[\sum_{\sigma\in C(T_n)}{(-1)^{\ell_D (\sigma)}y^{\onsp (\sigma)}}=\sum_{\sigma\in C(U_n)}{(-1)^{\ell (\sigma)}y^{\oinv (\sigma)}},\]
which together with \eqref{eq:fact} and Lemma~\ref{lem:unim} yields the result. \end{proof}
Setting $y=x$ gives the known result for the signed distribution of the odd length over $D_n$ (cf.\ \cite[Theorem 4.1 and Corollary 4.2]{BC2}).
\begin{cor}
Let $n\geq 2$. Then
\[\sum_{\sigma\in W(\Phi(D_n))}{(-1)^{\ell (\sigma)}x^{L_{\Phi(D_n)} (\sigma)}}= \left(\sum_{\sigma\in W(\Phi(A_{n-1})) }{(-1)^{\ell (\sigma)}x^{L_{\Phi(A_{n-1})} (\sigma)}} \right)^2.\]
\end{cor}
\section{Signed generating functions for Weyl groups}
We summarize in this section the results obtained for the signed generating functions of the odd length on the classical Weyl groups, and we record some computations that we made for the exceptional types. The signed generating functions for the exceptional types were computed with the Python package \verb|PyCox| (see \cite{Ge}).
\begin{thm}
Let $\Phi$ be a crystallographic root system. Then
\begin{equation*}\sum_{\sigma \in W(\Phi)}{(-1)^{\ell(\sigma)}x^{L_{\Phi}(\sigma)}}=\begin{cases}
\prod\limits_{i=2}^{n} \left(1+(-1)^{i-1}x^{\left\lfloor\frac{i}{2}\right\rfloor}\right), &\quad\mbox{if $\Phi$ is of type $A_{n-1}$,}\\ 
\prod\limits_{i=1}^{n} (1-x^{i}), &\quad\mbox{if $\Phi$ is of type $B_n$,}\\ 
 (1-x^{\left\lceil\frac{n}{2}\right\rceil})\prod\limits_{i=1}^{\left\lceil\frac{n}{2}\right\rceil} (1-x^{2i})^2, &\quad\mbox{if $\Phi$ is of type $C_n$,}\\
 \prod\limits_{i=2}^{n}(1+(-1)^{i-1}x^{\left\lfloor\frac{i}{2}\right\rfloor})^2,  &\quad\mbox{if $\Phi$ is of type $D_n$}.
 \end{cases}\end{equation*}
Moreover,\begin{align*}
\sum_{\sigma\in W(\Phi)}(-1)^{\ell(\sigma)}x^{L_{\Phi}(\sigma)}&=(1-x^2)^2(1-x^4)^2, &\mbox{if $\Phi$ is of type $F_4$},\\
\sum_{\sigma\in W(\Phi)}(-1)^{\ell(\sigma)}x^{L_{\Phi}(\sigma)}&=(1-x^2)(1-x^4)(1-x^6)(1-x^8),&\mbox{if $\Phi$ is of type $E_6$}, \\
\sum_{\sigma\in W(\Phi)}(-1)^{\ell(\sigma)}x^{L_{\Phi}(\sigma)}&=\prod_{i=2}^{8}(1-x^i),&\mbox{if $\Phi$ is of type $E_7$}.
\end{align*}
\end{thm}

It is conceivable, and we believe, that the generating function $\sum_{\sigma\in W(\Phi)}(-1)^{\ell(\sigma)}x^{L_{\Phi}(\sigma)}$
also factors nicely in type $E_8$. However, we have been unable to carry out this computation with the computing
resources at our disposal.

\begin{rmk} We record here the functions used to compute the generating functions with \verb|PyCox|.

\verb|def f(n):|

\verb|  W  = coxeter("W",n)|
  
\verb|  y  = var('y')|
  
\verb|  A  = allcoxelms(W)|
  
\verb|  Or = [i for i in range(W.N) if mod(sum((W.roots[i])),2)==1]|
  
\verb|  B  = [W.coxelmtoperm(A[i][j]) for i in range(len(A)) for j in range(len(A[i]))]|
  
\verb|  return sum(y^(W.permlength(v)) * x^(oddr(v,W.N,Or)) for v in B)|

\verb|def oddr(v,n,Or): |

\verb|  return sum(1 for j in Or if v[j]>n-1)|
\end{rmk}

\end{document}